\documentclass[12pt,reqno]{amsart}
\usepackage{amsmath}
\usepackage{amssymb}
\usepackage{amstext}
\usepackage{a4wide}
\usepackage{txfonts}
\usepackage{graphicx}
\usepackage{bm}
\usepackage[numbers,sort&compress]{natbib}
\allowdisplaybreaks \numberwithin{equation}{section}
\usepackage{color}
\usepackage{cases}

\numberwithin{equation}{section}

\newtheorem{theorem}{Theorem}[section]
\newtheorem{proposition}[theorem]{Proposition}

\newtheorem{lemma}[theorem]{Lemma}

\newtheorem*{Yudovich's Theorem}{Yudovich's Theorem}

\theoremstyle{definition}

\newtheorem{definition}[theorem]{Definition}

\theoremstyle{remark}
\newtheorem{remark}[theorem]{Remark}

\begin{document}

\title
[Stability of a class of exact solutions of the incompressible Euler equation]{Stability of a class of exact solutions of the incompressible Euler equation in a disk}

 \author{Guodong Wang}

\address{School of Mathematical Sciences, Dalian University of Technology, Dalian 116024, PR China}
\email{gdw@dlut.edu.cn}


\begin{abstract}
We prove a sharp orbital stability result for a class of exact steady solutions, expressed in terms of Bessel functions of the first kind, of the two-dimensional  incompressible Euler equation in a disk.  A special case of these solutions is the truncated Lamb dipole, whose stream function corresponds to the second eigenfunction of the Dirichlet Laplacian. The proof is achieved by establishing a suitable variational characterization for these solutions via conserved quantities of the Euler equation and employing a compactness argument.
\end{abstract}

\maketitle
\tableofcontents
\section{Introduction and main result}\label{sec1}

 In the mathematical theory of two-dimensional ideal fluids, stability analysis is an important and challenging research topic. Over the past century, extensive research has been conducted on various stability problems, which have resulted in many stability criteria for general flows \cite{A1,A2,BAR,FSV,LZSM,LZCMP,WTAMS,WG1,WG2}, as well as  a large number of stability results for specific flows of physical significance \cite{Abe,BG,CWCV,CWJMFM,CJ,CG,LLZ,WJDE,WMA}. Although substantial progress has been made in this field, a comprehensive understanding of the intrinsic mechanisms of hydrodynamic stability is still elusive. For most flows, even those with an explicit expression, determining their stability can still be highly complicated.

In this paper, we focus on a class of specific flows with explicit expressions for an ideal fluid in a two-dimensional disk. The stream/vorticity functions of these flows can be expressed in terms of Bessel functions of the first kind, which form a three-dimensional linear space.  In particular, these flows include the truncated Lamb dipole as a special case,   the stream/vorticity function of which is the second eigenfunction of the Laplace operator with zero Dirichlet boundary condition. We prove that these flows are Lyapunov stable up to two-dimensional rigid rotations with respect to the $L^p$ norm of the vorticity for any $1<p<\infty$. To the best of our knowledge, this appears to be the first stability result for two-dimensional ideal flows related to Dirichlet Laplacian eigenfunctions of order greater than 1 in a bounded domain.

Throughout this paper, let $D$ be the unit disk centered at the origin
 \[D=\left\{\mathbf x=(x_1,x_2)\in \mathbb R^2\mid |\mathbf x|:=\sqrt{x_1^2+x_2^2}<1\right\}.\]
Consider the Euler equation of a two-dimensional ideal fluid in $D$ with impermeability boundary condition:
\begin{equation}\label{euler}
\begin{cases}
  \partial_t\mathbf{v}+(\mathbf{v}\cdot\nabla)\mathbf{v}=-\nabla P, & t\in\mathbb R,\,\,\mathbf x\in D, \\
  \nabla\cdot\mathbf{v}=0, & t\in\mathbb R,\,\,\mathbf x\in D, \\
  \mathbf{v}\cdot\mathbf{n} =0,&   t\in\mathbb R,\,\,\mathbf x\in \partial D,
\end{cases}
\end{equation}
where $\mathbf{v}(t,\mathbf x)=(v_1(t,\mathbf x),v_2(t,\mathbf x))$ is the velocity field, $P(t,\mathbf x)$ is the scalar pressure, and $\mathbf{n}(\mathbf x)$ is the outward unit normal of $\partial D$ (i.e., $\mathbf n(\mathbf x)=\mathbf x$). The boundary condition means that there is no matter flow passing through $\partial D$.
Introduce the vorticity $\omega:= \partial_{x_1}v_2-\partial_{x_2}v_1$. Then $\omega$ satisfies the vorticity equation:
\begin{equation}\label{vor}
\begin{cases}
\partial_t\omega+\mathbf v\cdot\nabla\omega=0, & t\in\mathbb R,\,\,\mathbf x\in D,\\
\mathbf v=\nabla^\perp\mathcal{G}\omega,
\end{cases}
\end{equation}
where  $\nabla^\perp:=(\partial_{x_2},-\partial_{x_1}),$ and $\mathcal G$ is the inverse of $-\Delta$ under the zero Dirichlet boundary condition, i.e.,
\begin{equation}\label{defg}
\begin{cases}
    -\Delta \mathcal G\omega=\omega, & \mathbf x\in D, \\
    \mathcal G\omega=0,&\mathbf x\in \partial D.
  \end{cases}
  \end{equation}
Note that $\mathcal G$ can be explicitly expressed as follows:
\[\mathcal{G} \omega(t,\mathbf x) =\int_{D}G(\mathbf x,\mathbf y)\omega(t,\mathbf y)d\mathbf y,\quad G(\mathbf x,\mathbf y):=-\frac{1}{2\pi}\ln|\mathbf x-\mathbf y|+\frac{1}{2\pi}\ln\left|\frac{\mathbf x}{|\mathbf x|}-|\mathbf x|\mathbf y\right|.\]
It is also useful to introduce the stream function $\psi$ of a fluid flow, a function that is constant along every streamline. Up to a constant,  we can take  $\psi=\mathcal G\omega$.

There are rich results on the global well-posedness of the two-dimensional incompressible Euler equation. See \cite{De,DM,Giga,Y} for example. In this paper,  we mainly work in the setting of Yudovich \cite{Y}.

\begin{Yudovich's Theorem} \label{thmbt}
For any $\omega_0\in L^\infty(D)$, there exists a unique $\omega \in L^\infty(\mathbb R;L^\infty(D))$  such that
\begin{itemize}
\item[(1)] $\omega$ solves the vorticity equation in the  sense of distributions, i.e.,
    \[
  \int_{\mathbb R}\int_D\omega\left(\partial_t\varphi+ \nabla^\perp \mathcal G\omega\cdot\nabla\varphi\right) d\mathbf xdt=0\quad \forall\,\varphi\in C_c^{\infty}(\mathbb R\times D);
  \]
\item[(2)] $\omega\in C(\mathbb R;L^r(D))$ for any $r\in[1,\infty),$ and $\omega(0,\cdot)=\omega_0$;
\item[(3)] $\omega(t,\cdot)\in \mathcal{R}_{\omega_0}$ for any $t\in\mathbb R$, where
     $\mathcal R_{\omega_0}$ denotes the rearrangement class of $\omega_0$,
\begin{equation}\label{drc1}
 \mathcal R_{\omega_0}:=\left\{ v: D\to\mathbb R\mid  |\{\mathbf x\in D\mid v(\mathbf x) >s\}|=|\{\mathbf x\in D\mid \omega_0(\mathbf x) >s\}|\,\,\forall\,s\in\mathbb R\right\},
\end{equation}
where $|\cdot|$ denotes the two-dimensional Lebesgue measure;
\item[(4)]  $E(\omega(t,\cdot))=E(\omega_0)$ for any $t\in\mathbb R,$
where $E$ denotes the kinetic energy of the fluid,
\begin{equation}\label{tweak2}
E(\omega(t,\cdot)):=\frac{1}{2}\int_D\omega(t,\mathbf x)\mathcal G\omega(t,\mathbf x) d\mathbf x.
\end{equation}
\end{itemize}
\end{Yudovich's Theorem}
The function  $\omega$ is called a \emph{Yudovich solution} of the vorticity equation \eqref{vor} with initial vorticity $\omega_0$.
The proof of Yudovich's Theorem can also be found in \cite{BAR,MB,MP}.

Given $\bar\omega\in C^1(\bar D)$, it is a steady solution  of \eqref{vor} if and only if $\nabla\mathcal G\bar\omega$ and $\nabla \bar\omega$  are parallel. A sufficient condition for this to hold is that $\bar\omega$ satisfies  the functional relationship
\begin{equation}\label{sct01}
\bar\omega=g(\mathcal G\bar\omega)
\end{equation}
for some $g\in C^1(\mathbb R).$ See \cite{CWPr} for a more general sufficient condition  in the setting of weak solutions. Note that \eqref{sct01} is in fact a semilinear elliptic equation from the viewpoint of the stream function $\bar\psi:=\mathcal G\bar\omega$:
\begin{equation}\label{seli01}
\begin{cases}
-\Delta\bar\psi=g(\bar\psi),&\mathbf x\in D,\\
\bar\psi=0,&\mathbf x\in\partial D.
\end{cases}
\end{equation}

In the paper, we are concerned with the case where $g$ is a linear function. In this case, \eqref{sct01} becomes
$\bar\omega=\lambda\mathcal G\bar\omega,$
 which corresponds to the Laplacian eigenvalue problem for $\bar\psi:$
\begin{equation}\label{egp01}
\begin{cases}
-\Delta\bar\psi=\lambda \bar\psi,&\mathbf x\in D,\\
\bar\psi=0,&\mathbf x\in\partial D.
\end{cases}
\end{equation}
Denote by $J_n$ the Bessel function of the first kind of order $n$,
 \[J_n(s)=\sum_{k=0}^\infty\frac{(-1)^k}{k!(k+n)!}
\left(\frac{s}{2}\right)^{2k+n},  \quad n=0,1,2,\cdot\cdot\cdot, \,\, s\in\mathbb R,  \]
and by $j_{n,k}$  the $k$-th positive zero of $J_n.$
Then the set of eigenvalues of \eqref{egp01} is
\[\left\{j^2_{n,k}\mid  n=0,1,2,\cdot\cdot\cdot, \,  k=1,2,3,\cdot\cdot\cdot\right\}.\]
Note that the first two eigenvalues are $j_{0,1}^2$ and $j_{1,1}^2.$
The eigenspace $\mathbf E_{n,k}$ related to the eigenvalue $j^2_{n,k}$ is given by
\[\mathbf E_{n,k}={\rm span}\left\{J_n(j_{n,k}r)\cos (n\theta) ,\,J_n(j_{n,k}r)\sin (n\theta) \right\},\quad n=0,1,2,\cdot\cdot\cdot, \,\,k=1,2,3,\cdot\cdot\cdot,\]
where   $(r,\theta)$ denotes the polar coordinates,
\[x_1=r\cos\theta,\quad x_2=r\sin\theta.\]
It is easy to see that $\mathbf E_{n,k}$ is one-dimensional if $n=0,$ and is two-dimensional if $n\ge 1.$

In this paper, we focus on a larger linear space  of steady solutions of the Euler equation than $\mathbf E_{n,k}$. Define
   \begin{equation}\label{defvnk}
   \mathbf V_{n,k}:={\rm span}\left\{J_0(j_{n,k}r),\,J_n(j_{n,k}r)\cos (n\theta) ,\,J_n(j_{n,k}r)\sin (n\theta) \right\},\quad n=0,1, \cdot\cdot\cdot, \,\,k=1,2, \cdot\cdot\cdot.
   \end{equation}
\begin{lemma}
The space $ \mathbf V_{n,k}$ can be characterized by
\[ \mathbf V_{n,k}=\left\{v\in L^2(D)\mid  v=j_{n,k}^2\mathcal Gv+cJ_0(j_{n,k}),\,c\in\mathbb R\right\}.\]
As a consequence, any  element in $\mathbf V_{n,k}$ is a steady solution of \eqref{vor}.
\end{lemma}
\begin{proof}
 For any $ v\in\mathbf V_{n,k}$, we can write $v=v_1+v_2$, where $v_1\in \mathbf E_{n,k}$ and  $v_2=cJ_0(j_{n,k}r)$ for some $c\in\mathbb R$. Since $v_1\in \mathbf E_{n,k}$, it holds that 
\begin{equation}\label{yunc1}
v_1=j_{n,k}^2\mathcal Gv_1.
 \end{equation}
Using the fact that  $J_n$ solves  Bessel's differential equation (cf.  \eqref{bsleq} in Section \ref{sec2}),
one can verify that $v_2$ satisfies 
$-\Delta v_2=j_{n,k}^2 v_2$ in $ D$ and $v_2=cJ_0(j_{n,k})$ on $\partial D$, 
 or equivalently,
\begin{equation}\label{yunc2}
v_2=j_{n,k}^2\mathcal Gv_2+cJ_0(j_{n,k}).
\end{equation}
From \eqref{yunc1} and \eqref{yunc2}, we deduce that $v=j_{n,k}^2\mathcal Gv+cJ_0(j_{n,k})$.
 
Conversely, if $v\in L^2(D)$ satisfies $v=j_{n,k}^2\mathcal Gv+cJ_0(j_{n,k})$ for some $c\in\mathbb R$, then one can check that
$w:=v-cJ_{0}(j_{n,k}r)$
satisfies  $w=j_{n,k}^2\mathcal Gw$ and $w|_{\partial D}=0$ by 
\eqref{yunc2}. Hence  $w\in\mathbf E_{n,k}$, and thus $v=w+cJ_0(j_{n,k}r)\in\mathbf V_{n,k}$.  \end{proof}

We are concerned with the Lyapunov stability of steady flows with  vorticity belonging to $\mathbf V_{n,k}$. Specifically, given a steady solution $\bar\omega\in\mathbf V_{n,k}$ of \eqref{vor},  if we add a ``small" perturbation to it and follow the nonlinear dynamics of the vorticity equation \eqref{vor}, we want to know whether the perturbed solution remains ``close" to $\bar\omega$ for all time. Of course, the terms ``small" and ``close" should be specified in a rigorous mathematical setting.
Since $D$ is a disk with a continuous rotational symmetry, in some situations one may need to consider \emph{orbital stability}, i.e., the stability of the orbit of some function under the action of the two-dimensional rotation group. For the following two special cases, stability is known.
\begin{itemize}
  \item  [(1)]  $\bar\omega\in\mathbf V_{0,1}$. In this case, $\bar\omega$ is radially monotone,  so stability holds by  \cite[Theorem 3]{BAR} (or by  a generalized version of Arnold's second stability theorem, cf.  \cite[Theorem 1.5]{WTAMS}).
  \item [(2)]  $\bar\omega\in {\rm span}\left\{J_{0}(j_{1,1}r)\right\}\subset \mathbf V_{1,1}$. Such  $\bar\omega$ is also radially monotone, thus stability still holds.
\end{itemize}
 
The main purpose of this paper is to study the stability of the steady solutions in $\mathbf V_{1,1}$. For simplicity, in the rest of this paper we shall denote
\begin{equation}\label{defj11}
j:=j_{1,1}\approx 3.831706,
\end{equation}
\begin{equation}\label{defv11}
\mathbf V:=\mathbf V_{1,1}={\rm span}\left\{J_0(jr),\,J_1(jr)\cos \theta,\,J_1(jr)\sin \theta \right\}.
\end{equation}
Introduce the orbit $\mathcal O_\zeta$ of a function $\zeta:D\to \mathbb R$ under the action of  the two-dimensional rotation group $\mathbb S\mathbb O(2)$:
\begin{equation}\label{deforb}
\mathcal O_\zeta:=\left\{\zeta(r,\theta+\beta)\mid\beta\in\mathbb R\right\},
\end{equation}
where $(r,\theta)$ still denotes the polar coordinates.
Denote by $\|\cdot\|_p$   the norm in $L^p(D).$

Our main result can be stated as follows.

\begin{theorem}\label{thmmain}
Let $1<p<\infty$ and $\bar\omega\in\mathbf V$ be fixed.
 Then $\bar\omega$ is orbitally stable under the dynamics of  the vorticity equation \eqref{vor} in the following sense:  for any $\varepsilon>0$, there exists some $\delta>0$, such that for any $\omega_0\in L^\infty(D)$, if
\[\|\omega_0-\bar\omega\|_p<\delta,\]
then for any $t\in\mathbb R$, there exists some $\tilde\omega\in\mathcal O_{\bar\omega}$  such that
\[\|\omega(t,\cdot)-\tilde\omega\|_p<\varepsilon,\]
where $\omega(t,\mathbf x)$ is the Yudovich solution of \eqref{vor} with initial vorticity $\omega_0$.
\end{theorem}

\begin{remark}
The conclusion in Theorem \ref{thmmain} can be rephrased as follows:
for any $\varepsilon>0$, there exists some $\delta>0$, such that for any Yudovich solution $\omega(t,\mathbf x)$  of \eqref{vor}, it holds that
\[\min_{v\in\mathcal O_{\bar\omega}}\|\omega(0,\cdot)-v\|_p<\delta\quad \Longrightarrow\quad  \min_{v\in\mathcal O_{\bar\omega}}\|\omega(t,\cdot)-v\|_p<\varepsilon\quad\forall\,t\in\mathbb R.\]
\end{remark}

 \begin{remark}
Stability in Theorem \ref{thmmain} actually  holds for more general perturbations than Yudovich solutions, as long as the vorticity belongs to $L^p(D)$ and
 the kinetic energy and the distribution function of the vorticity are conserved. See Theorem \ref{thm42} in Section \ref{sec5} for the details.
 \end{remark}

 In  Section \ref{sec6}, we will show that the orbital stability  in Theorem \ref{thmmain} is actually sharp, in the sense that $\mathcal O_{\bar\omega}$ is the ``smallest" set containing $\bar\omega$ that is stable under the Euler dynamics.

As a straightforward consequence of Theorem \ref{thmmain}, we  obtain the orbital stability of the steady solutions in $\mathbf E_{1,1}$.  As far as we know, this is the first stability result on plane ideal flows related to the second-lowest eigenfunctions of the Dirichlet Laplacian in a bounded domain.
We also mention that the steady solutions in $\mathbf E_{1,1}$ are related to the Lamb dipole. The Lamb dipole, also called the Chaplygin-Lamb dipole \cite{Chap,Lamb}, is an exact steady solution of the incompressible Euler equation in the whole plane. For the Lamb dipole, the stream function $\psi_L$ can be explicitly expressed as follows (cf.  \cite{Abe,Lamb}):
\[\psi_L(r,\theta)=\begin{cases}
\frac{2W}{\mu J_0(j)}J_1(\mu r)\sin\theta,&r\leq \frac{j}{\mu },\\
W\left(r-\frac{j^2}{\mu^2 r}\right)\sin\theta,&r>\frac{j}{\mu},
\end{cases}\]
where $(W,0)$ with $W\in\mathbb R\setminus\{0\}$ is the velocity field of the fluid at space infinity, and $\mu$ is a positive parameter representing the strength of the vortex.  The corresponding vorticity $\omega_L$ is
\begin{equation}\label{ogal}
\omega_L(r,\theta)=\begin{cases}
\frac{2\mu W}{J_0(j)}J_1(\mu r)\sin\theta,&r\leq \frac{j}{\mu },\\
0,&r>\frac{j}{\mu}.
\end{cases}
\end{equation}
The orbital stability of the Lamb dipole was recently solved by Abe and Choi \cite{Abe}. See also \cite{WTAMS2} for another form of orbital stability.
Taking $\mu=j$ in \eqref{ogal}, we see that inside the separatrix $\{\psi_L=0\}$ (which is exactly the unit circle) the vorticity function $\omega_L$ is an element of $\mathbf E_{1,1}$. Therefore Theorem \ref{thmmain}  actually confirms the orbital stability of
the truncated Lamb dipole.

The most important step in proving  Theorem \ref{thmmain} is to show that $\mathcal O_{\bar\omega}$ has a clean variational characterization in terms of  conserved quantities of the Euler equation. More precisely,

\begin{theorem}\label{thmmain2}
For any $\bar\omega\in\mathbf V,$  $\mathcal O_{\bar\omega}$ is exactly the set of maximizers of $E$ relative to $\mathcal R_{\bar\omega}.$
\end{theorem}

The above variational characterization is obtained by establishing an energy-enstrophy inequality (cf. Section \ref{sec3}) and analyzing the equimeasurable partition of the three-dimensional linear space $\mathbf{V}$ (cf. Section \ref{sec4}).
Once  Theorem \ref{thmmain2} has been proved,  the desired orbital stability can be obtained by a standard compactness argument (cf. Section \ref{sec5}).

 Theorems \ref{thmmain} and \ref{thmmain2} are closely related to Burton's work \cite{BAR}. In \cite{BAR}, inspired by Arnold's pioneering work \cite{A1,A2}, Burton  established a very general stability criterion for two-dimensional ideal fluids in a bounded domain. More specifically,  Burton proved that \emph{any isolated maximizer of the kinetic energy $E$ relative to the  rearrangement class of some $L^p$ function is stable with respect to the $L^p$ norm of the vorticity, where $1<p<\infty$.} Recently, the author  \cite{WMA} extended Burton's result to the case of an isolated set of local maximizers. Burton's criterion is elegant and concise and, moreover, it provides an appropriate framework for studying the stability of two-dimensional ideal fluids. Many classical stability results can be reproved or generalized in the framework of Burton's criterion, including Arnold's second stability theorem \cite{A1,A2} and some of its extensions \cite{WTAMS,WG1,WG2}. Burton's stability criterion is also applied to the following:
\begin{itemize}
\item[(1)]  flows with  nonnegative and radially decreasing vorticity in a disk. (The vorticity is the unique global maximizer of the  kinetic energy relative to its rearrangement class,  cf. \cite{BMcL}.)

  \item [(2)] flows related to least energy solutions of the Lane-Emden equation in a convex domain. (The vorticity is the unique global maximizer of the  kinetic energy relative to its rearrangement class, cf. \cite{WJDE}.)
  \item[(3)]  some concentrated vortex flows constructed in \cite{EM, Tu}.  (The vorticity is an isolated local maximizer of the  kinetic energy relative to its rearrangement class, cf. \cite{CWCV,CWJMFM,WMA}.)
  \end{itemize}
Returning to Theorem \ref{thmmain2}, it provides a new nontrivial example of global maximizers of the kinetic energy relative to some rearrangement class, i.e., the $\mathbb S\mathbb O(2)$-orbit of  functions in $\mathbf V.$  Accordingly, Theorem \ref{thmmain} can be regarded as a new application of Burton's stability criterion.

We note that the stability of steady flows with vorticity $\bar\omega\in\mathbf V_{n,k}$ when $n\ge 2$ or $k\ge 2$ remains unclear, even when $\bar\omega$ is radial.
This seems to be quite a challenging problem. We are not sure whether Burton's stability criterion can be applied for these flows.

 This paper is organized as follows. In Section \ref{sec2}, we provide some preliminaries that will be used in subsequent sections. In Section \ref{sec3}, we 
 prove an energy-enstrophy inequality.
In Sections  \ref{sec4} and \ref{sec5}, we give the proofs of Theorem \ref{thmmain2} and  \ref{thmmain}, respectively.  Section \ref{sec6} is devoted to a brief discussion on rotating Euler flows in a disk.

\section{Preliminaries}\label{sec2}

First we recall some basic facts about Bessel functions of the first kind.
 \begin{lemma}[{\cite[p. 93 and p. 102]{Bo}}]\label{lemabesl}
 Let $n$ be a nonnegative integer, and $J_n$ be the Bessel function of the first kind of order $n$.
 \begin{itemize}
   \item [(1)] $J_n$  solves Bessel's differential equation, i.e.,
\begin{equation}\label{bsleq}
s^2J_n''(s)+sJ_n'(s)+(s^2-n^2)J_n(s)=0 \quad\forall\, s\in\mathbb R.
\end{equation}
   \item[(2)]
  The following recurrence relations hold:
   \begin{equation}\label{bpr1}
   (s^{n+1}J_{n+1})'=s^{n+1}J_{n},
   \end{equation}
     \begin{equation}\label{bpr2}
    (s^{-n}J_n)'=-s^{-n}J_{n+1},
    \end{equation}
   \begin{equation}\label{bpr3}
     s(J_{n}+J_{n+2})=2(n+1)J_{n+1}.
   \end{equation}
    In particular, taking $n=0$ in \eqref{bpr1} and \eqref{bpr2} we have
      \begin{equation}\label{bp12}
      (sJ_1)'=s J_{0},
      \end{equation}
          \begin{equation}\label{bp12f}
      J_0'=- J_{1}.
      \end{equation}
     Taking $n=0$ and $s=j$ (recall that $j$ is the first positive zero of $J_1$) in \eqref{bpr3} we have
       \begin{equation}\label{bp22}
       J_2(j)=-J_0(j).
       \end{equation}
   \item [(3)] If   $\alpha,\beta$ are two zeros of $J_n$, then
\begin{equation*}
  \int_0^1J_n(\alpha s)J_n(\beta s)sds=
  \begin{cases}
0, & \mbox{if }  \alpha\neq \beta,\\
\frac{1}{2}J^2_{n+1}(\alpha), &  \mbox{if }  \alpha=\beta.
\end{cases}
\end{equation*}
In particular, taking $n=1$ and $\alpha=\beta=j$ we obtain
  \begin{equation}\label{bp32}
              \int_0^1J_1^2(js) sds=\frac{1}{2}J^2_2(j).
              \end{equation}
 \end{itemize}
 \end{lemma}

Recall that $\mathcal G$ is the inverse of $-\Delta$ with zero Dirichlet boundary condition (cf. \eqref{defg} in Section \ref{sec1}).  Some basic facts about $\mathcal G$ are summarized as follows.

\begin{lemma}\label{pg0}
Let $1<p<\infty$ be fixed.
\begin{itemize}
    \item [(1)] $\mathcal{G}$ is a bounded operator from $L^{p}(D)$ to $W^{2,p}(D)$ and a compact operator from $L^{p}(D)$ to $L^{r}(D)$ for any $1 \leq r \leq \infty$.
    \item[(2)]$\mathcal G$ is symmetric, i.e.,
\[
\int_{D}v_1\mathcal G v_2 d\mathbf x=\int_{D}v_2\mathcal G v_1 d\mathbf x\quad\forall\, v_1,v_2\in {L}^{p}(D).
\]
\item[(3)]$\mathcal G$ is positive definite, i.e.,
\[
\int_{D}v\mathcal Gv d\mathbf x \geq 0  \quad\forall\,v\in {L}^{p}(D),
\]
and the  equality holds  if and only if $v\equiv 0.$
\end{itemize}
\end{lemma}

According to the above lemma, we see that the kinetic energy $E$ in \eqref{tweak2} is well-defined, strictly convex, and weakly sequentially continuous in $L^p(D)$ for any $1<p<\infty.$

The following two lemmas concerning the properties of rearrangement class  will be used in Section \ref{sec5}.

\begin{lemma}[{\cite[Theorem 4]{BMA}}]\label{lem202}
 Let $1<p<\infty$ be fixed. Let $q:=p/(p-1)$ be the H\"older conjugate of $p$.  Let $\mathcal R_1,$ $\mathcal R_2$ be two rearrangement classes  of some $L^p$ function and some $L^{q}$ function in $D$, respectively. Then for any $v_2\in\mathcal R_2,$ there exists $v_1\in \mathcal R_1$, such that
\[\int_{D} v_1v_2 d\mathbf x\geq \int_{D}w_1w_2d\mathbf x \quad\forall\, w_1\in {\mathcal R}_1,\,\,w_2\in {\mathcal R}_2.\]
\end{lemma}

\begin{lemma}[{\cite[Lemma 2.3]{BACT}}]\label{lem203}
  Let $1<p<\infty$ be fixed.  Let  $\mathcal R_1$ and $\mathcal R_2$ be two  rearrangement classes  of two $L^{p}$ functions  in $D$. Then for any $v_1\in\mathcal R_1$, there exists $v_2\in \mathcal R_2$ such that
\begin{equation}\label{kd01}
\|v_1-v_2\|_{p}=\inf\left\{\|w_1-w_2\|_{p}\mid  w_1\in\mathcal R_1,\, w_2\in\mathcal R_2 \right\}.
\end{equation}
\end{lemma}

 Consider the following Laplacian eigenvalue problem:
 \begin{equation}\label{ep823}
 \begin{cases}
 -\Delta u=\lambda u, & \mathbf x\in D,\\
\int_D ud\mathbf x=0,\\
u\mbox{ is constant on } \partial D.
 \end{cases}
 \end{equation}
  The following two lemmas will be used in Section \ref{sec3}.

 \begin{lemma}[{ \cite[Proposition 3.4]{GL}}]\label{lemagl}
The least eigenvalue of \eqref{ep823} (i.e., the smallest real number such that \eqref{ep823} has a nontrivial solution) is $j^2$, and the associated eigenspace   (i.e., the set of all solutions to \eqref{ep823} with $\lambda=j^2$) is $\mathbf V.$
 \end{lemma}

 \begin{lemma}\label{lemmainsert}
For any $v\in L^2(D)$,  $v\in\mathbf V$ if and only if $v=j^2(\mathcal Gv- I_{\mathcal Gv})$, where $I_{\mathcal Gv}$ denotes  the integral mean of $\mathcal Gv$  over $D$, i.e.,
\begin{equation}\label{integralmean}
I_{\mathcal Gv}:=\frac{1}{|D|}\int_D\mathcal Gv d\mathbf x.
\end{equation}

 \end{lemma}
\begin{proof}
Denote $u=\mathcal Gv- I_{\mathcal Gv}$. If $v=j^2(\mathcal Gv- I_{\mathcal Gv})$,  then $u$ satisfies
\[ \begin{cases}
 -\Delta u=j^2 u, & \mathbf x\in D,\\
\int_D ud\mathbf x=0,\\
u\mbox{ is constant on } \partial D.
 \end{cases}
\]
In other words, $u$ is an eigenfunction of \eqref{ep823} with eigenvalue $\lambda=j^2$. 
Thus $u\in \mathbf V$ by Lemma \ref{lemagl}, which implies that $v=-\Delta u=j^2u\in \mathbf V.$
Conversely, 
if $v\in\mathbf V$, then $-\Delta v=j^2 v$ and $I_v=0.$
 By applying $\mathcal G$ on both sides of  $-\Delta v=j^2 v$, we get $v=j^2\mathcal Gv+c.$  Using $I_v=0$,  we deduce that $c=-j^2I_{\mathcal Gv}$.  Hence  $v=j^2(\mathcal Gv-I_{\mathcal Gv}).$   

\end{proof}

\section{An energy-enstrophy inequality}\label{sec3}

The purpose of this section is to prove the following proposition.

\begin{proposition}\label{propsec3}
For any $v\in L^2(D)$ such that $\int_Dvd\mathbf x=0$, it holds that
\begin{equation}\label{ineqprop}
\int_Dv\mathcal Gv d\mathbf x\leq \frac{1}{j^2}\int_Dv^2d\mathbf x.
\end{equation}
Moreover, the equality is attained if and only if $v\in\mathbf V$.

\end{proposition}
 
For a fluid flow, the integral of the square of the vorticity is referred to as \emph{enstrophy}, which is commonly used to measure the intensity of vorticity. The above proposition indicates that the kinetic energy can be controlled by a constant multiple of the enstrophy.

 To prove Proposition \ref{propsec3}, we 
 consider the following  variational problem:
\[
M:=\sup_{v\in\mathcal C}\int_Dv\mathcal Gvd\mathbf x,\tag{$V$}\]
where
\[\mathcal C:=\left\{v\in L^2(D)\mid \|v\|_2=1,\,\int_Dv d\mathbf x=0\right\}.\]

 \begin{lemma}\label{lema23}
  There exists a maximizer to $(V)$,  and any maximizer $\tilde v$ satisfies
\begin{equation}\label{pflm01}
\tilde v=\frac{1}{M}\left(\mathcal G\tilde v-I_{\mathcal G\tilde v}\right),
\end{equation}
where  $I$   denotes the integral mean as in \eqref{integralmean}. 
\end{lemma}
\begin{proof}

By  Lemma \ref{pg0}, it is easy to see that $M$ is real and positive.
Let $\{v_n\} \subset\mathcal C$ be a sequence such that
\[\lim_{n\to\infty}\int_Dv_n\mathcal Gv_nd\mathbf x=M.\]
Up to a subsequence, we can assume that $v_n$ converges weakly to some $\tilde  v$ in $L^2(D)$. It is clear that
\[\|\tilde  v\|_2\leq 1,\quad \int_D\tilde vd\mathbf x=0.\]
Moreover, since $E$ is weakly sequentially continuous in $L^{2}(D)$, we have that
\[\int_D\tilde v\mathcal G\tilde vd\mathbf x=M.\]
In particular, $\tilde  v\not\equiv  0.$ To show that $\tilde  v$ is a maximizer, it suffices to prove that $\|\tilde  v\|_2=1.$  Suppose by contradiction that $\mu:=\|\tilde  v\|_2<1$, then  $\hat v:= \mu^{-1} {\tilde v}\in\mathcal C$.  Therefore
\[\int_D\hat v\mathcal G\hat vd\mathbf x\leq M,\]
 which  contradicts
\[\int_D\hat v\mathcal G\hat vd\mathbf x=\frac{1}{\mu^2}\int_D\tilde v\mathcal G\tilde vd\mathbf x=\frac{M}{\mu^2}>M.\]
Hence we have proved the existence of a maximizer.

Next we show that any maximizer $\tilde v$ satisfies \eqref{pflm01}. For any  $\phi\in L^2(D)$ satisfying $I_{\phi}=0$, define a family of test functions $\{v_\varepsilon\}\subset \mathcal C$ by setting
\[v_\varepsilon:=\frac{\tilde v+\varepsilon\phi}{\|\tilde v+\varepsilon\phi\|_{2}},\]
where $\varepsilon\in\mathbb R$ is small in absolute value such that $\|\tilde v+\varepsilon \phi\|_{2}>0$.  Then
\[\frac{d}{d\varepsilon}\int_Dv_\varepsilon\mathcal G v_\varepsilon d\mathbf x\bigg|_{\varepsilon=0}=0,\]
which, by a direct computation,  yields
\begin{equation}\label{iidty01}
 \int_D\phi\mathcal G\tilde vd\mathbf x=M\int_D\tilde v\phi d\mathbf x.
 \end{equation}
 Here we used the symmetry of $\mathcal G$.
Since \eqref{iidty01} holds for any $\phi\in L^2(D)$ such that $I_\phi=0$, we deduce that
\begin{equation}\label{iidty02}
 \int_D\left(\varphi-I_\varphi\right)\mathcal G\tilde vd\mathbf x=M\int_D\tilde v\left(\varphi-I_\varphi\right) d\mathbf x \quad \forall\,\varphi\in L^2(D).
 \end{equation}
 Using  $I_{\tilde v}=0$,  \eqref{iidty02} can be written as
\[
 \int_D\left(\mathcal G\tilde v-I_{\mathcal G\tilde v }\right)\varphi d\mathbf x=M\int_D\tilde v \varphi d\mathbf x \quad \forall\,\varphi\in L^2(D),
\]
from which we can conclude \eqref{pflm01}.
\end{proof}

 \begin{lemma}\label{lema24}
 $M=j^{-2}.$
\end{lemma}

\begin{proof}

Let  $v\in\mathcal C$ be a maximizer to $(V)$. By Lemma \ref{lema23},  $u:=\mathcal G v-I_{\mathcal G v}$ satisfies $-\Delta u= M^{-1}u$ in $D$. It is also clear that $I_u=0$, and that $u$  is constant on $\partial D.$ In other words, $u$  is an eigenfunction of  \eqref{ep823} with eigenvalue $\lambda=M^{-1}$. Since $j^2$ is the least eigenvalue of \eqref{ep823} by Lemma \ref{lemagl}, we conclude that  $M^{-1}\geq  j^{2}$, or equivalently,  $M\leq j^{-2}$.

  To prove the inverse inequality, take $v\in \mathbf V$ such that $\|v\|_2=1$. Then it is clear that $v\in\mathcal C$, and $v=j^2(\mathcal G  v-I_{\mathcal G v})$ by Lemma \ref{lemmainsert}. 
  By integration by parts,  
\[
 1=\int_D v^2d\mathbf x=j^2\int_D v(\mathcal G  v-I_{\mathcal G v}) d\mathbf x=j^2\int_D  v\mathcal G  v d\mathbf x\leq j^2 M,
 \] 
 which yields $M\geq j^{-2}$.
\end{proof}

\begin{lemma}\label{lema26}
The set of maximizers to $(V)$ agrees with $\mathbf V\cap \mathcal B$, where
 \[\mathcal B:=\left\{v\in L^2(D)\mid \|v\|_2=1\right\}.\]
\end{lemma}
\begin{proof}
 If $v\in\mathcal C$ is a maximizer to $(V)$, then $v\in\mathcal B$, and 
 $v=j^2(\mathcal Gv-I_{\mathcal Gv})$  by Lemmas \ref{lema23} and \ref{lema24}.   Taking into account  Lemma \ref{lemmainsert}, we get $v\in\mathbf V.$
 Thus $v\in \mathbf V\cap \mathcal B.$

 Conversely, if $v\in\mathbf V\cap \mathcal B,$ then $v\in\mathcal C$, and $v=j^2(\mathcal Gv-I_{\mathcal Gv})$ by Lemma \ref{lemmainsert}.   By integration by parts, 
  \[1=\int_Dv^2d\mathbf x= j^2\int_Dv(\mathcal Gv-I_{\mathcal Gv})d\mathbf x=j^2\int_Dv\mathcal Gvd\mathbf x.\]
So $v$ is a maximizer to ($V$) by Lemma \ref{lema24}.
 \end{proof}

Now we are ready to prove Proposition \ref{propsec3}.
\begin{proof}[Proof of Proposition \ref{propsec3}]
Fix $v\in L^2(D)$ such that $I_v=0.$
Without loss of generality, we assume that $v\not\equiv 0.$ Then $v/\|v\|_2 \in \mathcal C$. By Lemma \ref{lema24}, it holds that
\begin{equation}\label{ineqlst}
\int_D\frac{v}{\|v\|_2}\mathcal G\left(\frac{v}{\|v\|_2}\right)d\mathbf x\leq M=\frac{1}{j^2},
\end{equation}
 which yields \eqref{ineqprop}. Moreover, in view of Lemma \ref{lema26},  the inequality in \eqref{ineqlst}  is an equality if and only $v/\|v\|_2\in\mathbf V\cap \mathcal B$,  which is equivalent to $v\in\mathbf V.$ 
\end{proof}

\section{Proof of Theorem \ref{thmmain2}}\label{sec4}

Throughout this section, let  $\bar\omega\in\mathbf V$ be fixed. Denote 
\begin{equation}\label{supbar}
\mathsf M:= \sup_{v\in  \mathcal R_{\bar\omega}}E(v),\quad 
\mathcal M:=\left\{v\in \mathcal R_{\bar\omega}\mid E(v)=\mathsf M\right\}.
\end{equation}
To prove Theorem \ref{thmmain2}, we need to show that $\mathcal M=\mathcal O_{\bar\omega}$.

 We begin with the following proposition, which is a consequence of Proposition \ref{propsec3}.

\begin{proposition}\label{propvcz} 
It holds that
\begin{equation}\label{mathsfm}
\mathsf M=\frac{1}{2j^2}\int_D\bar\omega^2 d\mathbf x
\end{equation}
and
\begin{equation}\label{mathcalm} 
\mathcal M= \mathbf V\cap\mathcal R_{\bar\omega}.
\end{equation}
\end{proposition}

 \begin{proof}
 First we prove \eqref{mathsfm}.
 For any $v\in \mathcal R_{\bar\omega},$ it holds that
  \[I_v=I_{\bar\omega}=0,\quad \int_Dv^2 d\mathbf x=\int_D\bar\omega^2d\mathbf x.\]
By Proposition \ref{propsec3},
\begin{equation*}
\mathsf M=\sup_{v\in\mathcal R_{\bar\omega}}E(v)=\frac{1}{2}\sup_{v\in\mathcal R_{\bar\omega}}\int_D  v\mathcal G  vd\mathbf x\leq \frac{1}{2j^2}\int_D v^2d\mathbf x=\frac{1}{2j^2}\int_D \bar\omega^2d\mathbf x
=\frac{1}{2}\int_D \bar\omega\mathcal G\bar\omega d\mathbf x\leq \mathsf M.
\end{equation*} 
So \eqref{mathsfm} holds. 

Next we prove \eqref{mathcalm}. If $v\in \mathbf V\cap\mathcal R_{\bar\omega}$, then by Proposition \ref{propsec3},
\[E(v)=\frac{1}{2}\int_Dv\mathcal Gv d\mathbf x=\frac{1}{2j^2}\int_Dv^2 d\mathbf x =\frac{1}{2j^2}\int_D\bar\omega^2 d\mathbf x =\mathsf M.\]
Thus $v\in \mathcal M.$ Conversely, if $v\in\mathcal M$, then $v\in\mathcal R_{\bar\omega}$ and $E(v)=\mathsf M,$ which implies that
\[\int_Dv\mathcal Gv d\mathbf x=2E(v)=2\mathsf M=\frac{1}{j^2}\int_D\bar\omega^2 d\mathbf x=\frac{1}{j^2}\int_Dv^2 d\mathbf x. \]
  By Proposition \ref{propsec3} again, we get $v\in\mathbf V,$ and thus $v\in \mathbf V\cap\mathcal R_{\bar\omega}$.
 \end{proof}

The rest of this section is devoted to proving
\begin{equation}\label{vreo}
\mathbf V\cap\mathcal R_{\bar\omega}=\mathcal O_{\bar\omega}.
\end{equation}
Before the proof, we give some comments on \eqref{vreo}. For two measurable functions $v_1,v_2:D\to\mathbb R,$ we say that $v_1$ and $v_2$ are equimeasurable if   $v_1\in\mathcal R_{v_2}$ (or $v_2\in\mathcal R_{v_1}$), denoted by  $v_1\sim v_2$. It is easy to check that $``\sim"$ is an equivalence relation (i.e.,  $``\sim"$ is reflexive, symmetric and transitive).
Then \eqref{vreo} can be rephrased as follows: for any $\bar\omega\in\mathbf V,$ under the equivalence relation  $``\sim"$, the equivalence class of $\bar\omega$ is exactly the set of two-dimensional rigid rotations of $\bar\omega$. In this sense, \eqref{vreo} provides a clean equivalence partition of $\mathbf V$ under the equivalence relation  $``\sim"$. We mention that the equimeasurable partition problem of a finite-dimensional function space also appeared in the study of other stability problems related to Laplacian eigenfunctions, such as the orbital stability of degree-2 Rossby-Haurwitz waves \cite{CWZ}, and the stability of sinusoidal flows on a flat 2-torus \cite{WZCV}.

Now we prove \eqref{vreo}. Since $\mathcal O_{\bar\omega}\subset\mathbf V\cap\mathcal R_{\bar\omega}$ is obvious, it suffices to verify the converse inclusion
\begin{equation}\label{vreo2}
\mathbf V\cap\mathcal R_{\bar\omega}\subset\mathcal O_{\bar\omega}.
\end{equation}
Without loss of generality, we can assume that $\bar\omega\not\equiv 0.$ Since $\bar\omega\in\mathbf V$, we can write
\[\bar\omega=aJ_0(jr)+bJ_1(jr)\cos(\theta+\beta),\]
where $a,b,\beta\in\mathbb R$ with $b\geq 0$.
Then
\[
\mathcal O_{\bar\omega}=\{aJ_0(jr)+bJ_1(jr)\cos(\theta+\gamma)\mid \gamma\in\mathbb R\}.
\]
For $v\in\mathbf V\cap\mathcal R_{\bar \omega}$ with the form
  \[ v=a'J_0(jr)+b'J_1(jr)\cos(\theta+\beta'),\quad a',\beta' \in\mathbb R,\, b'\geq 0,\]
 in order to prove $v\in\mathcal O_{\bar\omega}$, it suffices to verify
 \begin{equation}\label{abeq1}
 a'=a,\quad b'=b.
 \end{equation}
 To this end, we need to make use of the fact that $v\in\mathcal R_{\bar \omega}$.  For any positive integer $m$,  it holds that
   \begin{equation}\label{abeq2}
   \int_Dv^m d\mathbf x=\int_D\bar\omega^m d\mathbf x.
   \end{equation}
 In terms of polar coordinates, \eqref{abeq2} becomes
    \begin{equation}\label{rmab1}
    \int_0^1\int_0^{2\pi}
  \left(a'J_0(jr)+b'J_1(jr)\cos\theta\right)^m rd\theta dr
  =\int_0^1\int_0^{2\pi}
  \left(aJ_0(jr)+bJ_1(jr)\cos\theta\right)^m rd\theta dr.
  \end{equation}
Taking $m=2$,   \eqref{rmab1} gives
     \begin{equation}\label{rmab2}
     p_1a'^2+p_2b'^2=p_1a^2+p_2b^2,
     \end{equation}
  where
      \begin{equation}\label{p12}
   p_1:=2\pi\int_0^1J_0^2(jr)rdr=\frac{2\pi}{j^2}\int_0^{j}J_0^2(r)rdr,
   \end{equation}
     \begin{equation}\label{p12f}
   p_2 := \pi \int_0^{1}J_1^2(jr)rdr=\frac{\pi}{j^2}\int_0^{j}J_1^2(r)rdr.
   \end{equation}
Taking $m=3$, \eqref{rmab1} gives
   \begin{equation}\label{rmab3}
   q_1a'^3+ q_2a'b'^2 =q_1a^3+ q_2ab^2,
   \end{equation}
  where
    \begin{equation}\label{q12}
   q_1 := 2\pi \int_0^{1}J_0^3(jr)rdr= \frac{2\pi}{j^2}\int_0^{j}J_0^3(r)rdr,
   \end{equation}
   \begin{equation}\label{q12f}
    q_2:= 3\pi\int_0^1J_0(jr)J_1^2(jr)rdr=\frac{3\pi}{j^2}\int_0^{j}J_0 (r)J_1^2(r)rdr.
   \end{equation}
   Note that we have used the following facts in obtaining \eqref{rmab2} and \eqref{rmab3}:
  \[\int_0^{2\pi}\cos\theta d\theta=0,\quad \int_0^{2\pi}\cos^3\theta d\theta=0.\]

  To proceed, we need to compute the integrals in \eqref{p12}, \eqref{p12f}, \eqref{q12} and \eqref{q12f}.  The following integral identities related to $J_0$ and $J_1$ are needed.

\begin{lemma}\label{lmap}
The following identities hold:
\begin{equation}\label{apd1}
 \int_0^sJ_0^2(\tau)\tau d\tau=\frac{1}{2}s^2\left(J^2_0(s)+J^2_1(s)\right)\quad \forall\,s\in\mathbb R,
\end{equation}
\begin{equation}\label{apd2}
 \int_0^1J_1^2(js) sds=\frac{1}{2}J^2_0(j),
\end{equation}
\begin{equation}\label{apd3}
 \int_0^sJ_0^3(\tau)\tau d\tau=sJ_0^2(s)J_1(s)+2\int_0^sJ_0(\tau)J_1^2(\tau)\tau d\tau\quad \forall\,s\in\mathbb R,
\end{equation}
\begin{equation}\label{apd4}
\int_0^sJ_0(\tau)J_1^2(\tau)\tau d\tau=\frac{1}{3}sJ_1^3(s)
    +\frac{2}{3}\int_0^sJ_1^3(\tau) d\tau\quad \forall\,s\in\mathbb R.
\end{equation}

\end{lemma}

\begin{proof}
First we prove \eqref{apd1}. Notice that both sides of \eqref{apd1} vanish at $s=0,$ so we only need to compare their derivatives.  By a direct computation,
\begin{align*}
\left(\frac{1}{2}s^2(J^2_0 +J^2_1 )\right)'&=s\left(J^2_0 +J^2_1 \right)
+ s^2\left(J_0J'_0+J_1J'_1 \right)\\
&=sJ^2_0+sJ^2_1
-s^2 J_0 J_1 +s J_1 (sJ_0 -J_1 )\\
&=sJ^2_0,
\end{align*}
Note that in the second equality we have used \eqref{bp12} and \eqref{bp12f}. Next,  \eqref{apd2} is a straightforward consequence of \eqref{bp22} and \eqref{bp32}.
Finally, to prove \eqref{apd3} and \eqref{apd4}, it suffices to notice that
\begin{align*}
         \left(sJ_0^2J_1\right)'&=J_0^2J_1+2sJ_0J_0'J_1+sJ_0^2J_1'\\
         &=J_0^2J_1-2sJ_0 J^2_1+ J_0^2(sJ_0-J_1)\\
         &=sJ_0^3-2sJ_0J_1^2,
         \end{align*}
and
\begin{align*}
         \left(sJ_1^3 \right)'&=J_1^3+3sJ_1^2J_1'\\
         &=J_1^3+3 J_1^2(sJ_0-J_1)\\
         &=3sJ_0J_1^2-2J_1^3,
\end{align*}
where the recurrence relations \eqref{bp12} and \eqref{bp12f} were  used again.
\end{proof}

  Back to \eqref{p12} and \eqref{q12}, applying Lemma \ref{lmap}, we have that
\[p_1 =\pi J_0^2(j),\quad p_2=\frac{\pi}{2} J_0^2(j),\]
\[q_1= \frac{8\pi}{3j^2}\int_0^jJ_1^3(\tau) d\tau,\quad q_2=\frac{2\pi}{j^2}\int_0^jJ_1^3(\tau) d\tau.\]
In particular, $p_1,p_2,q_1,q_2$ are all positive (notice that $J_1$ is positive in $(0,j)$) and satisfy the following relations:
\begin{equation}\label{pqrl}
p_1=2p_2,\quad q_1=\frac{4}{3}q_2.
\end{equation}
From \eqref{rmab2}, \eqref{rmab3} and \eqref{pqrl}, we find that
  $(a',b')$ and $(a,b)$ are two solutions of
\begin{equation}\label{sae}
   \begin{cases}
         2x^2+y^2=r_1, \\
       4x^3+3xy^2=r_2
     \end{cases}
\end{equation}
 for some $r_1,r_2\in\mathbb R.$
To prove \eqref{abeq1}, it suffices to show that \eqref{sae} has at most one real solution such that $y\geq 0$.

\begin{lemma}\label{lemar12}
For any $r_1,r_2\in\mathbb R$, the system of algebraic equations \eqref{sae} has at most one real solution such that $y\geq 0$.
\end{lemma}
\begin{proof}
Without loss of generality, we can assume that $r_1>0,$ since otherwise the system has only one solution $(0,0)$.
 Inserting the first equation into the second one yields
\begin{equation}\label{efx}
-2x^3+3r_1 x=r_2.
\end{equation}
Since we are only concerned with real solutions satisfying $y\geq 0$, it suffices to show that \eqref{efx} has at most one real solution in the interval $2x^2\leq r_1$.  Observing that  $-2x^3+3r_1 x$ is strictly increasing in the interval $2x^2\leq r_1$, we obtain the desired result.
\end{proof}

\begin{proof}[Proof of Theorem \ref{thmmain2}]
From the above discussion, we see that \eqref{vreo} holds. Theorem \ref{thmmain2} is a straightforward consequence of \eqref{mathcalm} and \eqref{vreo}.
\end{proof}

\section{Proof of Theorem \ref{thmmain}}\label{sec5}

In this section, we give the proof of Theorem \ref{thmmain} based   Theorem \ref{thmmain2}. Although the basic idea follows from Burton's paper \cite{BAR},
  there is still something new in our proof.  In  \cite{BAR}, when dealing with perturbed solutions that do not belong to the rearrangement class, Burton constructed a ``follower" for the perturbed solution by solving a linear transport equation.  To ensure that a suitable weak solution of the linear transport equation exists,   some additional regularity on  the perturbed solutions are needed. In the following proof, we provide a new approach to construct the ``follower"  based on Lemma \ref{lem203}, thereby reducing the requirements on  the perturbed solutions.

Throughout this section, let $1<p<\infty$ be fixed. To see clearly that our stability result actually holds for very general perturbations, we introduce the notion of admissible maps.

\begin{definition}\label{defadms}
   An admissible map is a map  $\zeta:\mathbb R\to  L^p(D)$ such that
 \[E(\zeta(t))=E(\zeta(0))\quad \forall\,t\in\mathbb R,\]
  \[\zeta(t)\in \mathcal R_{\zeta(0)}\quad \forall\,t\in\mathbb R.\]
\end{definition}

Note that any Yudovich solution $\omega(t,\mathbf x)$ of the vorticity equation \eqref{vor} corresponds to an admissible map $\zeta(t):=\omega(t,\cdot)$.

The following theorem is a general version of Theorem \ref{thmmain}.

\begin{theorem}\label{thm42}
Let $\bar\omega\in\mathbf V$ be fixed. Then for any $\varepsilon>0,$ there exists some $\delta>0,$ such that for any  admissible map $\zeta(t)$, it holds that
\[\min_{v\in\mathcal O_{\bar\omega}}\|\zeta(0)-v\|_p<\delta\quad \Longrightarrow \quad \min_{v\in\mathcal O_{\bar\omega}}\|\zeta(t)-v\|_p<\varepsilon\quad\forall\,t\in\mathbb R.\]
\end{theorem}

The following compactness result is crucial in proving  Theorem \ref{thm42}.
\begin{proposition}\label{pcomp}
Let  $\bar\omega\in\mathbf V$ be fixed. Let  $\{v_n\} \subset\mathcal R_{\bar\omega}$ be a sequence satisfying
\begin{equation}\label{bnr1}
\lim_{n\to\infty}E(v_n)=E(\bar\omega).
\end{equation}
Then $\{v_n\}$ has a subsequence converging to some $\tilde v\in\mathcal O_{\bar\omega}$ strongly in $L^p(D)$.

\end{proposition}
\begin{proof}

Since $\{v_n\}$ is bounded in $L^p(D),$ we can assume, up to a subsequence, that $v_n$ converges weakly to some $\tilde v\in \overline{\mathcal R^w_{\bar\omega}}$ in $L^p(D)$, where $\overline{\mathcal R^w_{\bar\omega}}$ denotes the weak closure of $ {\mathcal R}_{\bar\omega}$ in $L^p(D)$. Since $E$ is sequentially weakly continuous in $L^{p}(D)$,   we deduce from \eqref{bnr1} that
\begin{equation}\label{re79}
E(\tilde v)= E(\bar \omega).
\end{equation}
If we can show  $\tilde v\in\mathcal R_{\bar\omega},$ then $\tilde v\in \mathcal O_{\bar\omega}$ according to Theorem \ref{thmmain2} and, moreover, $v_n$ converges  to  $\tilde v$  strongly in $L^p(D)$ as $n\to\infty$ by uniform convexity (cf. Proposition 3.32 of \cite{Bre}).

The rest of the proof is devoted to proving  $\tilde v\in\mathcal R_{\bar\omega}.$ Since $\bar\omega$ is a maximizer of $E$ relative to $\mathcal R_{\bar\omega}$,
we have from \eqref{re79} that
\begin{equation}\label{re81}
E(\tilde v)\geq E(v)\quad\forall\,v\in \mathcal R_{\bar\omega}.
\end{equation}
By Lemma \ref{lem202} in Section 2,  there exists some $\hat v\in\mathcal R_{\bar\omega}$ such that
\[
\int_{D} v\mathcal G\tilde vd\mathbf x\leq \int_{D}\hat v\mathcal G\tilde vd\mathbf x\quad\forall\,v\in\mathcal R_{\bar\omega},
\]
and thus
\[
\int_{D} v\mathcal G\tilde vd\mathbf x\leq \int_{D}\hat v\mathcal G\tilde vd\mathbf x\quad\forall\,v\in\overline{\mathcal R^w_{\bar\omega}}.
\]
In particular,
\begin{equation}\label{reny9}
\int_{D} \tilde v\mathcal G\tilde vd\mathbf x\leq \int_{D}\hat v\mathcal G\tilde v  d\mathbf x.
\end{equation}
Now we compute as follows:
\begin{equation}\label{compp}
\begin{split}
 E(\tilde v-\hat v)&=E(\tilde v)+E(\hat v)-\int_{D} \hat v\mathcal G\tilde  vd\mathbf x\\
&\leq E(\tilde v)+E(\hat v)-\int_{D} \tilde v\mathcal G\tilde vd\mathbf x\\
&=E(\hat v)-E(\tilde v)\\
&\leq 0.
\end{split}
\end{equation}
Note that in the first equality of \eqref{compp} we used the symmetry of $\mathcal G$, in the first inequality we used \eqref{reny9}, and in the last inequality we used \eqref{re81}.
Taking into account the fact that $\mathcal G$ is positive definite (cf. Lemma \ref{pg0}), we obtain  $\tilde v=\hat v.$  In particular, $\tilde v \in\mathcal R_{\bar\omega}.$

\end{proof}

Now we are ready to prove Theorem \ref{thm42}.

\begin{proof}[Proof of Theorem \ref{thm42}]

 It suffices to show that for any sequence of admissible maps $\{\zeta_n(t)\}$ and any sequence of times $\{t_n\}\subset\mathbb R$,  if
\begin{equation}\label{tt09}
\lim_{n\to\infty}\|\zeta_n(0)-\hat\omega\|_{p}=0
\end{equation}
for some $\hat \omega\in\mathcal O_{\bar\omega}$,
then up to a subsequence there exists some $\tilde\omega\in\mathcal O_{\bar\omega}$ such that
\begin{equation}\label{tt010}
\lim_{n\to\infty}\|\zeta_n(t_n)-\tilde\omega\|_{p}= 0.
\end{equation}
By \eqref{tt09} and the weak continuity of $E$ in $L^p(D)$,
 we have
\[\lim_{n\to\infty}E(\zeta_n(0))=E(\hat\omega)=E(\bar\omega),\]
which, in conjunction with the fact that $E$ is conserved for admissible maps, implies that
\begin{equation}\label{gx21}
\lim_{n\to\infty}E(\zeta_n(t_n))=E(\bar\omega).
\end{equation}

With \eqref{gx21} at hand, we can easily obtain the desired stability \emph{when the perturbations are restricted on $\mathcal R_{\bar\omega}.$} As a matter of fact,  if  $\zeta_n(0)\in\mathcal R_{\bar\omega}$, then $\zeta_n(t)\in\mathcal R_{\bar\omega}$ for any $t\in\mathbb R$ by the definition of admissible maps. Thus \eqref{tt010} holds for some $\tilde\omega\in\mathcal O_{\bar\omega}$ by Proposition \ref{pcomp}.

To deal with general perturbations, we need to construct a sequence of ``followers" $\{\eta_n\}\subset\mathcal R_{\bar\omega}$ associated to the sequence $\{\zeta_n(t_n)\}$. Specifically, for each positive integer $n$, we can take some  $\eta_n\in\mathcal R_{\bar\omega}$ such that
\[\|\eta_n-\zeta_n(t_n)\|_{p}=\inf\left\{\|u-v\|_{p}\mid u\in\mathcal R_{\bar\omega},\,v\in\mathcal R_{\zeta_n(t_n)}\right\}.\]
Such $\eta_n$ exists by Lemma \ref{lem203} in Section \ref{sec2}.
 In particular, it holds that
\begin{equation}\label{ty12} \|\eta_n-\zeta_n(t_n)\|_{p}\leq \|\hat\omega-\zeta_n(0)\|_{p},
\end{equation}
where we used the facts that $\hat\omega\in\mathcal R_{\bar\omega}$ and $\zeta_n(0)\in\mathcal R_{\zeta_n(t_n)}$.
From \eqref{tt09} and \eqref{ty12}, we obtain
\begin{equation}\label{tt121}
\lim_{n\to\infty}\|\eta_n-\zeta_n(t_n)\|_{p} =0,
\end{equation}
and thus
\begin{equation}\label{tt08}
\lim_{n\to\infty} E(\eta_n)= E(\bar\omega)
\end{equation}
by \eqref{gx21}.
That is, we have constructed  a sequence  $\{\eta_n\}\subset\mathcal R_{\bar\omega}$ such that \eqref{tt08} holds.
Applying  Proposition \ref{pcomp} to the sequence $\{\eta_n\}$, we find that $\eta_n$ converges, up to a subsequence, to some $\tilde\omega\in\mathcal O_{\bar\omega}$ strongly in $L^p(D)$  as $n\to\infty$, which together with \eqref{tt121} leads to \eqref{tt010}.

\end{proof}

\section{Rotating solutions}\label{sec6}

First we prove a basic fact for the vorticity equation \eqref{vor} in a disk.
\begin{lemma}\label{lema51}
  Let $\Omega\in\mathbb R$ be fixed. Then, in polar coordinates,
 $\omega(t,r,\theta)$ is a Yudovich solution of  \eqref{vor} if and only if $\omega(t,r,\theta-\Omega t)+2\Omega$ is a Yudovich solution of  \eqref{vor}.
\end{lemma}
\begin{proof}
First we expressed the vorticity equation \eqref{vor} in polar coordinates.
In view of the relation
\[x_1=r\cos\theta,\quad x_2=r\sin\theta,\]
we have from the chain rule that
\[\partial_{r}\omega =  \partial_{x_1}\omega\cos\theta+\partial_{x_2}\omega\sin\theta,\quad \partial_{\theta}\omega =  r\left(-\partial_{x_1}\omega  \sin\theta+\partial_{x_2}\omega \cos \theta\right),\]
which implies that
\begin{equation}\label{crul01}
\partial_{x_1}\omega=\partial_r\omega\cos\theta-\frac{1}{r}\partial_\theta \omega\sin\theta,\quad \partial_{x_2}\omega=\partial_r\omega\sin\theta+\frac{1}{r}\partial_\theta \omega\cos\theta.
\end{equation}
From \eqref{crul01}, we see that \eqref{vor} in polar coordinates has the form
\begin{equation}\label{vpc01}
 \partial_t\omega+  \frac{1}{r}\left(\partial_\theta\mathcal G\omega\partial_r\omega-\partial_r\mathcal G\omega\partial_\theta\omega\right)=0.
\end{equation}
 Denote $\zeta(t,r,\theta)=\omega(t,r,\theta-\Omega t)+2\Omega.$ Then, by the chain rule, we have that
\begin{equation}\label{pdr01}
\partial_t\zeta =\left[\partial_t\omega
-\Omega\partial_\theta\omega\right]\big|_{(t,r,\theta-\Omega t)},
\end{equation}
\begin{equation}\label{pdr02}
\partial_r\zeta = \partial_r\omega\big|_{(t,r,\theta-\Omega t)},\quad \partial_r \mathcal G\zeta = \partial_r\mathcal G\omega\big|_{(t,r,\theta-\Omega t)}-\Omega r,
\end{equation}
\begin{equation}\label{pdr03}
\partial_\theta\zeta(t,r,\theta)= \partial_\theta\omega\big|_{(t,r,\theta-\Omega t)},\quad \partial_\theta \mathcal G\zeta = \partial_\theta\mathcal G\omega\big|_{(t,r,\theta-\Omega t)},
\end{equation}
where we used the facts that
\[\mathcal G\Omega=\frac{1}{4}\Omega\left(1-r^2\right),\quad \mathcal G(\omega(t,r,\theta- \Omega t))=\mathcal G\omega(t,r,\theta- \Omega t).\]
From \eqref{pdr01}-\eqref{pdr03}, we find that $\zeta$ solves \eqref{vpc01} if and only if $\omega$ solves \eqref{vpc01}.

\end{proof}

As an application of Lemma \ref{lema51}, we can show that the orbital stability we have obtained in Theorem \ref{thmmain} is in fact optimal.

\begin{theorem}\label{thm52}
Let $\bar\omega\in\mathbf V$ be fixed. Then there exists a sequence of solutions   $\{\omega_n\}$ of the vorticity equation such that
 \begin{itemize}
   \item [(1)] $\omega_n(0,\cdot)\to \bar\omega$ in $L^p(D)$ as $n\to\infty$ and,
   \item  [(2)]for any $\tilde \omega\in\mathcal O_{\bar\omega},$ there exists a sequence  of times $\{t_n\}$ such that
     $\omega_n(t_n,\cdot)\to \tilde\omega$ in $L^p(D)$ as $n\to\infty$.
 \end{itemize}
\end{theorem}
\begin{proof}
 It suffices to choose
 \[\omega_n(t,r,\theta)=\bar\omega\left(r,\theta-\frac{1}{n}t\right)
 +\frac{2}{n},\quad n=1,2,3,\cdot\cdot\cdot.\]
  Note that $\omega_n$ solves \eqref{vor} by Lemma \ref{lema51}.
\end{proof}

Denote
\[\mathbf V+2\Omega:=\left\{v+2\Omega \mid v\in\mathbf V\right\}.\]
According to Lemma \ref{lema51}, the Yudovich solution with initial vorticity  in $\mathbf V+2\Omega$ rotates around the origin with angular velocity $\Omega$, thus forming an $\mathbb S\mathbb O(2)$-orbit. The following theorem is about the stability of these orbits.

\begin{theorem}\label{thm53}
 Let $1<p<\infty$ be fixed. Let $\bar\omega\in\mathbf V$ and  $\Omega\in\mathbb R$ be fixed.  Then for any $\varepsilon>0$, there exists some $\delta>0$, such that for any Yudovich solution $\omega(t,\mathbf x)$ of the vorticity equation \eqref{vor}, it holds that
\[\min_{v\in\mathcal O_{\bar\omega+2\Omega}}\|\omega(0,\cdot)-v\|_p<\delta\quad \Longrightarrow\quad  \min_{v\in\mathcal O_{\bar\omega+2\Omega}}\|\omega(t,\cdot)-v\|_p<\varepsilon\quad\forall\,t\in\mathbb R.\]
\end{theorem}

\begin{proof}
It suffices to show that for any sequence of Yudovich solutions $\{\omega_n(t,\mathbf x)\}$ of the vorticity equation \eqref{vor}, if
\begin{equation}\label{tg09}
\lim_{n\to\infty}\|\omega_n(0,\cdot)-\hat v\|_{p}=0
\end{equation}
for some $\hat v\in\mathcal O_{\bar\omega+2\Omega}$,
then up to a subsequence there exists some $\tilde v\in\mathcal O_{\bar\omega+2\Omega}$ such that
\begin{equation}\label{tg010}
\lim_{n\to\infty}\|\omega_n(t_n,\cdot)-\tilde v\|_{p}= 0.
\end{equation}
Denote $\hat \omega:=\hat v-2\Omega$. Then $\hat\omega\in\mathcal O_{\bar\omega}$ and
\begin{equation}\label{tl09}
\lim_{n\to\infty}\|\omega_n(0,\cdot)-2\Omega-\hat \omega\|_{p}=0
\end{equation}
Denote
\[\xi_n(t,r,\theta):=\omega_n(t,r,\theta+\Omega t)-2\Omega.\]
Since $\omega_n(t,r,\theta)$ is a solution of \eqref{vor}, using Lemma \ref{lema51} we see that $\xi_n(t,r,\theta)$ is a solution of \eqref{vor} with initial value $\xi_n(0,\cdot)=\omega_n(0,\cdot)-2\Omega.$ Therefore by Theorem \ref{thmmain}, we deduce that up to a subsequence
\begin{equation}\label{tl10}
\lim_{n\to\infty}\|\xi_n(t_n,\cdot)-\check \omega\|_{p}=0
\end{equation}
for some $\check \omega\in\mathcal O_{\bar\omega}.$ From \eqref{tl10},  it is easy to check that $\omega_n(t_n,\cdot)$ converges to  $\tilde\omega+2\Omega$ for some $\tilde\omega\in\mathcal O_{\bar\omega}$ in $L^p(D)$ as $n\to\infty.$
\end{proof}

\bigskip

 \noindent{\bf Acknowledgements:} 
G. Wang was supported by NNSF of China grant  12471101.

\bigskip
\noindent{\bf  Data Availability} Data sharing not applicable to this article as no datasets were generated or analysed during the current study.

\bigskip
\noindent{\bf Declarations}

\bigskip
\noindent{\bf Conflict of interest}  The author declare that they have no conflict of interest to this work.

\phantom{s}
 \thispagestyle{empty}

\end{document}